\documentclass[oneside,english]{amsart}
\usepackage[latin9]{inputenc}
\usepackage{amsthm}
\usepackage{amstext}
\usepackage{amssymb}
\usepackage{esint}
\usepackage[all]{xy}

\makeatletter
\numberwithin{equation}{section}
\numberwithin{figure}{section}
\theoremstyle{plain}
\newtheorem{thm}{\protect\theoremname}[section]
  \theoremstyle{plain}
  \newtheorem{conjecture}[thm]{\protect\conjecturename}
  \theoremstyle{definition}
  \newtheorem{defn}[thm]{\protect\definitionname}
  \theoremstyle{remark}
  \newtheorem{rem}[thm]{\protect\remarkname}
  \theoremstyle{plain}
  \newtheorem{prop}[thm]{\protect\propositionname}
  \theoremstyle{plain}
  \newtheorem{lem}[thm]{\protect\lemmaname}
  \theoremstyle{plain}
  \newtheorem{cor}[thm]{\protect\corollaryname}
  \theoremstyle{definition}
  \newtheorem{example}[thm]{\protect\examplename}

\makeatother

\usepackage{babel}
  \providecommand{\conjecturename}{Conjecture}
  \providecommand{\corollaryname}{Corollary}
  \providecommand{\definitionname}{Definition}
  \providecommand{\examplename}{Example}
  \providecommand{\lemmaname}{Lemma}
  \providecommand{\propositionname}{Proposition}
  \providecommand{\remarkname}{Remark}
\providecommand{\theoremname}{Theorem}

\begin{document}

\title[Mass formulas and quotient singularities]{Mass formulas for local Galois representations and quotient singularities
I: a comparison of counting functions}

\author{Melanie Machett Wood}

\address{Department of Mathematics, University of Wisconsin-Madison, 480 Lincoln
Drive, Madison, WI 53705 USA, and American Institute of Mathematics,
360 Portage Ave, Palo Alto, CA 94306-2244 USA}

\email{mmwood@math.wisc.edu}

\author{Takehiko Yasuda}

\address{Department of Mathematics, Graduate School of Science, Osaka University,
Toyonaka, Osaka 560-0043, Japan}

\email{takehikoyasuda@math.sci.osaka-u.ac.jp}

\maketitle
\global\long\def\AA{\mathbb{A}}
\global\long\def\PP{\mathbb{P}}
\global\long\def\NN{\mathbb{N}}
\global\long\def\GG{\mathbb{G}}
\global\long\def\ZZ{\mathbb{Z}}
\global\long\def\QQ{\mathbb{Q}}
\global\long\def\CC{\mathbb{C}}
\global\long\def\FF{\mathbb{F}}
\global\long\def\LL{\mathbb{L}}
\global\long\def\RR{\mathbb{R}}

\global\long\def\bx{\mathbf{x}}
\global\long\def\bf{\mathbf{f}}
\global\long\def\ba{\mathbf{a}}
\global\long\def\bs{\mathbf{s}}
\global\long\def\bt{\mathbf{t}}
\global\long\def\bw{\mathbf{w}}
\global\long\def\bb{\mathbf{b}}
\global\long\def\bv{\mathbf{v}}
\global\long\def\bp{\mathbf{p}}

\global\long\def\cN{\mathcal{N}}
\global\long\def\cW{\mathcal{W}}
\global\long\def\cY{\mathcal{Y}}
\global\long\def\cM{\mathcal{M}}
\global\long\def\cF{\mathcal{F}}
\global\long\def\cX{\mathcal{X}}
\global\long\def\cE{\mathcal{E}}
\global\long\def\cJ{\mathcal{J}}
\global\long\def\cO{\mathcal{O}}
\global\long\def\cD{\mathcal{D}}
\global\long\def\cZ{\mathcal{Z}}
\global\long\def\cR{\mathcal{R}}
\global\long\def\cC{\mathcal{C}}
\global\long\def\cG{\mathcal{G}}

\global\long\def\fs{\mathfrak{s}}
\global\long\def\fp{\mathfrak{p}}
\global\long\def\fm{\mathfrak{m}}

\global\long\def\Spec{\mathrm{Spec}\,}
\global\long\def\Hom{\mathrm{Hom}}

\global\long\def\Var{\mathrm{Var}}
\global\long\def\Gal{\mathrm{Gal}}
\global\long\def\Jac{\mathrm{Jac}}
\global\long\def\Ker{\mathrm{Ker}}
\global\long\def\Im{\mathrm{Im}}
\global\long\def\Aut{\mathrm{Aut}}
\global\long\def\st{\mathrm{st}}
\global\long\def\diag{\mathrm{diag}}
\global\long\def\characteristic{\mathrm{char}}
\global\long\def\tors{\mathrm{tors}}
\global\long\def\sing{\mathrm{sing}}
\global\long\def\red{\mathrm{red}}
\global\long\def\Ind{\mathrm{Ind}}
\global\long\def\nr{\mathrm{nr}}
 \global\long\def\univ{\mathrm{univ}}
\global\long\def\length{\mathrm{length}}
\global\long\def\sm{\mathrm{sm}}
\global\long\def\top{\mathrm{top}}
\global\long\def\rank{\mathrm{rank}}
\global\long\def\Mot{\mathrm{Mot}}
\global\long\def\age{\mathrm{age}\,}
\global\long\def\et{\mathrm{et}}
\global\long\def\hom{\mathrm{hom}}
\global\long\def\tor{\mathrm{tor}}

\global\long\def\Conj#1{\mathrm{Conj}(#1)}
\global\long\def\Mass#1{\mathrm{Mass}(#1)}
\global\long\def\Inn#1{\mathrm{Inn}(#1)}
\global\long\def\bConj#1{\mathbf{Conj}(#1)}
\global\long\def\Hilb{\mathrm{Hilb}}
\global\long\def\sep{\mathrm{sep}}
\global\long\def\GL#1#2{\mathrm{GL}_{#1}(#2)}
\global\long\def\codim{\mathrm{codim}}

\begin{abstract}
We study a relation between the Artin conductor and the weight coming
from the motivic integration over wild Deligne-Mumford stacks. As
an application, we prove some version of the McKay correspondence,
which relates Bhargava's mass formula for extensions of a local field
and the Hilbert scheme of points.
\end{abstract}

\section{Introduction}

The paper was motivated by an observation that two formulas appearing
in very different subjects look quite similar. One is Bhargava's mass
formula \cite{MR2354798} for a weighted count of \'etale algebras
(up to isomorphism) of fixed degree over a given local field. The
other is an explicit formula for a generating function concerning
the Hilbert scheme of points, which is essentially due to Ellingsrud
and Str\o mme \cite{MR870732}. In these two formulas, similar polynomials
appear and both have partition numbers as coefficients. 

According to Kedlaya \cite{MR2354797}, Bhargava's formula is interpreted
as a formula about counting local Galois representations. Let $K$
be a local field with residue field $\FF_{q}$ and $G_{K}$ its absolute
Galois group. Now the formula becomes of the form,
\begin{equation}
\frac{1}{n!}\cdot\sum_{\rho\in\Hom_{\text{cont.}}(G_{K},S_{n})}q^{-\ba(\rho)}=\sum_{m=0}^{n-1}P(n,n-m)\cdot q^{-m}.\label{eq:mass formula intro}
\end{equation}
Here $\rho$ runs over continuous homomorphisms $G_{K}\to S_{n}$
and $\ba(\rho)$ is the Artin conductor of the induced representation
$G_{K}\to\GL n{\CC}$. Then, as in \cite{MR2354797,MR2411405}, it
is natural to look for formulas when the group $S_{n}$ and the function
$-\ba$ in the left hand side are replaced with something else.

Let $H$ be the Hilbert scheme of $n$ points of $\AA_{\FF_{q}}^{2}$
and $Z\subset H$ the locus parametrizing subschemes supported at
the origin. The cell decomposition of $Z$ in \cite{MR870732,MR2492446}
shows that 
\begin{equation}
\sharp Z(\FF_{q})=\sum_{m=0}^{n-1}P(n,n-m)\cdot q^{m}.\label{eq:hilb formula intro}
\end{equation}
One finds that the right hand side is the same as that of (\ref{eq:mass formula intro})
except signs of exponents of $q$.

The aim of this paper is to make clearer the relation between formulas
(\ref{eq:mass formula intro}) and (\ref{eq:hilb formula intro}).
The key ingredient is the wild McKay correspondence studied in \cite{MR3230848,Yasuda:2013fk}.
Let $K$ be a local field with residue field $\kappa=\FF_{q}$, $\cO_{K}$
its integer ring and $X:=\AA_{\cO_{K}}^{n}/\Gamma$ the quotient scheme
of $\AA_{\cO_{K}}^{n}$ by an $\cO_{K}$-linear faithful action of
a finite group $\Gamma$ without pseudo-reflections. Suppose that
there exists a crepant resolution $Y\to X$ and let $Z\subset Y$
be the preimage of the origin of $X_{\kappa}:=X\otimes_{\cO_{K}}\kappa$.
\[
\xymatrix{ &  & \AA_{\cO_{K}}^{n}\ar[d]^{/\Gamma}\\
Z\subset Y\ar[rr]^{\text{crep. res.}} &  & 0\in X
}
\]
In this setting, we will pose the following variant of what was conjectured
in \cite{Yasuda:2013fk} in the motivic context and in the case of
algebraically closed residue field:
\begin{conjecture}[Conjecture \ref{conj: McKay points}]
We have
\[
\sharp Z(\FF_{q})=\frac{1}{\sharp\Gamma}\cdot\sum_{\rho\in\Hom_{\mathrm{cont.}}(G_{K},\Gamma)}q^{\bw(\rho)}.
\]
Here $\bw(\rho)$ is the weight of $\rho$ (for details, see Section
\ref{sec:Weights}).
\end{conjecture}
This was proved in \cite{MR3230848} when $\kappa$ is a finite field
of characteristic $p$, $K=\kappa((t))$, $\Gamma=\ZZ/p\ZZ$ and the
$\Gamma$-action on $\AA_{\kappa[[t]]}^{n}$ is defined over $\kappa$.
A particularly interesting situation of the conjecture is when $Y$
is the Hilbert scheme of $n$ points, $H_{\cO_{K}}$, now defined
over $\cO_{K}$ instead of $\FF_{q}$, $X$ is the symmetric product
of the affine plane, $S^{n}\AA_{\cO_{K}}^{2}=\AA_{\cO_{K}}^{2n}/S_{n}$
and $\Gamma$ is $S_{n}$. Then the left hand side of the equality
in the conjecture is nothing but that of formula (\ref{eq:hilb formula intro})
and the right is the left hand side of (\ref{eq:mass formula intro})
with $-\ba$ replaced with $\bw$. Therefore it is natural to ask
what the relation between the Artin conductor and the weight is. We
will study it in the tame case and also for permutation representations,
whether tame or wild. Using an obtained relation and Bhargava's more
precise formula \cite{MR2354798}, we will prove the following:
\begin{thm}[Theorem \ref{thm: Hilb}]
\label{thm:main intro}The conjecture above holds for $Y=H_{\cO_{K}}$,
$X=S^{n}\AA_{\cO_{K}}^{2}$ and $\Gamma=S_{n}$.
\end{thm}
Also we will prove the conjecture when $K$ has equal characteristic
prime to $\sharp\Gamma$ and the $\Gamma$-action is defined over
the residue field (Corollary \ref{cor: point tame 2}).

From Theorem \ref{thm:main intro}, we immediately see that formulas
for counts of local Galois representations with respect to $-\ba$
and $\bw$ interchange by the replacement $q\leftrightarrow q^{-1}$
in this particular case. We will find this duality in a few other
cases too. It is still very mysterious why and when the duality holds
especially in the wild case. This problem will be discussed in the
subsequent paper \cite{Wood-Yasuda-II}.

The paper is organized as follows. In Sections 2 to 4, we treat counting
functions such as $\ba$, $\bw$ and their relatives, and discuss
their relations. The last and longest section, Section 5, is devoted
to the McKay correspondence. Here we briefly recall the background
and formulate several versions of the McKay correspondence both in
the tame and wild cases, some of which are still conjectural. We prove
some of them in some cases. To formulate and prove a version of the
McKay correspondence in the tame case over an arbitrary perfect field,
we make heavy use of Deligne-Mumford stacks. Since this part of the
paper is rather independent of the rest, the reader unfamiliar with
stacks may skip it.

\subsection{Convention}

A local field means a complete discrete valuation field with (possibly
infinite) perfect residue field, unless otherwise noted.

\subsection{Acknowledgements}

Wood was supported by American Institute of Mathematics Five-Year
Fellowship and National Science Foundation grants DMS-1147782 and
DMS-1301690. Yasuda was supported by Grants-in-Aid for Scientific
Research (22740020).

\section{Total masses and the Artin conductor}

\subsection{Total masses and counting functions}

Let $K$ be a local field and $G_{K}:=\Gal(K^{\sep}/K)$ its absolute
Galois group. For a finite group $\Gamma,$ we put $S_{K,\Gamma}$
to be the set of continuous homomorphisms $G_{K}\to\Gamma.$ Given
a function $c:S_{K,\Gamma}\to\RR,$ the \emph{total mass} of $(K,\Gamma,c)$
is 
\[
M(K,\Gamma,c):=\frac{1}{\sharp\Gamma}\sum_{\rho\in S_{K,\Gamma}}q^{-c(\rho)}\in\RR_{\ge0}\cup\{\infty\},
\]
which was defined by Kedlaya \cite{MR2354797} for a special choice
of $c$ and Wood \cite{MR2411405} in general. This quantity is the
main concern of this paper. We call $c$ a \emph{counting function. }

We are mainly interested in two choices of the counting function $c$.
One is the \emph{Artin conductor} as studied in \cite{MR2354797}
and the other the \emph{weight} originating in the study of motivic
integration over wild Deligne-Mumford stacks \cite{Yasuda:2013fk}.%
\footnote{When $\Gamma$ is constructed from symmetric groups by using wreath
products and direct products, there is yet another interesting choice
of the counting function \cite{MR2411405}.%
} Both are associated to a representation of $\Gamma.$ They and their
relatives share many properties. For the unified treatment, we will
introduce the following notion.
\begin{defn}
\label{def: counting system}For each local field $K$, let us denote
by $R_{K}$ either a fixed field (independent of $K$), $K$ itself
or its integer ring $\cO_{K}$. A \emph{counting system }is a family
of functions,
\[
c_{\bullet}=(c_{K,\Gamma,\tau}:S_{K,\Gamma}\to\RR)_{K,\Gamma,\tau},
\]
where $K$ runs over local fields, $\Gamma$ finite groups and $\tau$
finite-dimensional $\Gamma$-representations $\Gamma\to\GL n{R_{K}}$. 
\end{defn}
Next we introduce several properties of a counting function system.
\begin{defn}
\label{def: complete}We say that $c_{\bullet}$ is \emph{additive
}if we always have 
\[
c_{K,\Gamma,\tau\oplus\sigma}=c_{K,\Gamma,\tau}+c_{K,\Gamma,\sigma}
\]
and if for the trivial representation $\alpha:\Gamma\to\GL 1{R_{K}}$,
we have
\[
c_{K,\Gamma,\alpha}\equiv0.
\]

We say that $c_{\bullet}$ is \emph{geometric }if we always have 
\[
c_{K,\Gamma,\tau}(\rho)=c_{K^{\nr},\Gamma,\tau^{\nr}}(\rho^{\nr}),
\]
where the superscript ``$\nr$'' means passing from $K$ to the
completion of the maximal unramified extension of $K$, which we denote
by $K^{\nr}$ by a slight abuse of notation.

We say that $c_{\bullet}$ is \emph{convertible }if for a homomorphism
$\phi:\Gamma_{1}\to\Gamma_{2}$ of finite groups and a $\Gamma_{2}$-representation
$\tau:\Gamma_{2}\to\GL n{R_{K}}$, we always have
\[
c_{K,\Gamma_{1},\tau\circ\phi}(\rho)=c_{K,\Gamma_{2},\tau}(\phi\circ\rho)\quad(\rho\in S_{K,\Gamma_{1}}).
\]

Finally we say that $c_{\bullet}$ is \emph{complete} if it is additive,
geometric and convertible.
\end{defn}

\subsection{The Artin conductor}

Let $L/K$ be a finite Galois extension of local fields with $G:=\Gal(L/K)$
and $G_{i}$, $i=0,1,\dots$, the $i$th ramification subgroups, which
are defined by
\[
G_{i}=\{g\in G\mid v_{L}(gx-g)\ge i+1\text{ for all }x\in\cO_{L}\}.
\]
Here $v_{L}$ is the normalized valuation of $L$. For a representation
$\tau:G\to\GL nk$ with $k$ either a fixed field or $K$, following
\cite{MR885783}, we define the \emph{Artin conductor }of $\tau$
by
\[
\ba_{\tau}(L/K)=\ba_{\tau}(L):=\sum_{i=0}^{\infty}\frac{1}{(G_{0}:G_{i})}\cdot\codim\,(k^{n})^{G_{i}}.
\]
We write the conductor as a function in $L$ for the consistency with
the later use. 
\begin{rem}
If $k=\CC$, then the definition above coincides with the one using
the inner product of characters as defined in \cite[VI, \S2]{MR554237}.
Taguchi \cite{MR1908909} proved that our Artin conductor satisfies
the same induction formula as the ordinary one over $\CC$ does. In
particular, the well-known relation between the Artin conductor and
the discriminant still holds in our setting.
\end{rem}
We also define the \emph{Swan conductor} and the \emph{tame part (of
the Artin conductor) }as follows:
\begin{align*}
\bs_{\tau}(L) & :=\sum_{i=1}^{\infty}\frac{1}{(G_{0}:G_{i})}\cdot\codim\,(k^{n})^{G_{i}}\text{ and}\\
\bt_{\tau}(L) & :=\ba_{\tau}(L)-\bs_{\tau}(L)=\codim\,(k^{n})^{G_{0}}.
\end{align*}

\subsection{The Artin conductor as a counting system}

Let $K$ be a local field and fix a representation $\tau:\Gamma\to\GL nk$
with $k=K$ or a field independent of $K$. Then for $\rho\in S_{K,\Gamma},$
we put $L_{\rho}/K$ to be the Galois extension corresponding to the
kernel of $\tau\circ\rho$ and $\tau':\Gamma'\hookrightarrow\GL nk$
to be the induced representation of $\Gamma':=\Im(\tau\circ\rho)$.
We define the function 
\begin{align*}
\ba_{K,\Gamma,\tau}=\ba_{\tau}:S_{K,\Gamma} & \to\ZZ_{\ge0}\\
\rho & \mapsto\ba_{\tau}(\rho):=\ba_{\tau'}(L_{\rho}),
\end{align*}
which we again call the \emph{Artin conductor. }Similarly we define
the \emph{Swan conductor,} $\bs_{K,\Gamma,\tau}=\bs_{\tau}$, and
the \emph{tame part of the Artin conductor,} $\bt_{K,\Gamma,\tau}=\bt_{\tau}$,
as functions on $S_{K,\Gamma}.$ We thus have obtained three counting
systems $\ba_{\bullet}$, $\bs_{\bullet}$ and $\bt_{\bullet}$. 
\begin{prop}
$\ba_{\bullet}$, $\bs_{\bullet}$ and $\bt_{\bullet}$ are complete.\end{prop}
\begin{proof}
The additivity and convertibility are obvious from the definition.
The geometricity follows from the fact that the higher ramification
groups $G_{i}$ of a Galois extension $L/K$ is stable by passing
to $L^{\nr}/K^{\nr}$. 
\end{proof}
The completeness is shared by all counting systems appearing in this
paper including the weight introduced below. Let us now see a property
of $ $the Artin conductor which distinguish it from the weight. 
\begin{lem}
Let $\alpha:\Gamma\to\Gamma$ be an automorphism preserving all subgroups
and define a map
\[
\alpha^{*}:S_{K,\Gamma}\to S_{K,\Gamma},\,\rho\mapsto\alpha\circ\rho.
\]
Then, for $c=\ba,$ $\bs$ and $\bt$ and for any $\Gamma$-representation
$\tau$, we have
\[
c_{\tau}=c_{\tau}\circ\alpha^{*}.
\]
\end{lem}
\begin{proof}
With the notation as above, the Galois extensions $L_{\rho}$ and
$L_{\alpha\circ\rho}$ are identical, which shows the lemma.
\end{proof}

\subsection{Bhargava's mass formula}

Let $K$ be a local field with a finite residue field $\FF_{q}.$
Serre \cite{MR500361} proved the following \emph{mass formula}: for
a positive integer $n$, 
\[
\sum_{L}\frac{1}{\sharp\Aut(L/K)}\cdot q^{-v_{K}(d_{L/K})}=q^{1-n},
\]
where $L/K$ runs over the isomorphism classes of totally ramified
field extensions with $[L:K]=n$ and $d_{L/K}$ denotes their discriminants.
Then Bhargava \cite{MR2354798} proved an analogous formula for \'etale
extensions: if $P(n,r)$ denotes the number of partitions of $n$
into exactly $r$ parts, then 
\[
\sum_{E}\frac{1}{\sharp\Aut(E/K)}\cdot q^{-v_{K}(d_{E/K})}=\sum_{m=0}^{n-1}P(n,n-m)\cdot q^{-m},
\]
where $E/K$ runs over the isomorphism classes of \'etale extensions
with $[E:K]=n.$

Let $S_{n}$ be the $n$th symmetric group. Each element $\rho\in S_{K,S_{n}}$
defines a $G_{K}$-action on $\{1,2,\dots,n\}$ and an \'etale extension
$E_{\rho}/K$ of degree $n$. This gives a one-to-one correspondence
between the classes of $\rho\in S_{K,S_{n}}$ modulo the conjugation
of $S_{n}$ and the isomorphism classes of \'etale extensions $E/K$
with $[E:K]=n.$ Let 
\[
\sigma:S_{n}\to\GL nk
\]
 be the defining representation of $S_{n}$ with $k$ an arbitrary
field. Kedlaya \cite{MR2354797} showed%
\footnote{In fact, he considers only the case $k=\CC$. However, since $\sigma$
is a permutation representation, $\ba_{\sigma}$ is independent of
$k$. Also, for any representation, Taguchi \cite{MR1908909} proved
that the relation between the Artin conductor and the discriminant
is valid over any field.%
} that for $\rho\in S_{K,S_{n}},$ 
\begin{equation}
\ba_{\sigma}(\rho)=v_{K}(d_{E_{\rho}/K})\label{eq:Artin=00003Ddisc}
\end{equation}
and that through the correspondence above, Bhargava's formula is rewritten
as follows: 
\begin{equation}
M(K,S_{n},\ba_{\sigma})=\sum_{i=0}^{n-1}P(n,n-m)\cdot q^{-m}.\label{eq:Bhargava-Kedlaya}
\end{equation}
We can slightly generalize Equality (\ref{eq:Artin=00003Ddisc}) as
follows:
\begin{lem}
\label{lem:Kedlaya}Let $\tau:\Gamma\to\GL nk$ be a permutation representation
of a finite group $\Gamma$. For $\rho\in S_{K,\Gamma},$ we denote
by $E_{\rho}/K$ the corresponding \'etale extension of degree $n.$
Then 
\[
\ba_{\tau}(\rho)=v_{K}(d_{E_{\rho}/K}).
\]
\end{lem}
\begin{proof}
From the convertibility of $\ba_{\bullet}$, we may replace $\Gamma$
with $S_{n}$ and $\tau$ with its defining representation $\sigma$
without changing either side of the equality. Thus we can reduce to
(\ref{eq:Artin=00003Ddisc}).
\end{proof}

\section{Weights\label{sec:Weights}}

Another function on $S_{K,\Gamma}$, the weight, came from the study
of the wild McKay correspondence using motivic integration \cite{Yasuda:2013fk}
(see Section 5). 

Let $\Gamma$ be a finite group. As $G_{K}$ is the \'etale fundamental
group of $\Spec K,$ there is a one-to-one correspondence between
$S_{K,\Gamma}$ and the set of pointed \'etale $\Gamma$-torsors
over $\Spec K$ modulo isomorphism (see \cite[page 44]{MR559531}).
As the data of pointing are not so important in what follows, we will
often ignore them. We denote the $\Gamma$-torsor corresponding to
$\rho\in S_{K,\Gamma}$ as 
\begin{equation}
T_{\rho}=\Spec M_{\rho}\to\Spec K.\label{eq:given torsor}
\end{equation}
Let $\cO_{M_{\rho}}$ be the integral closure of $\cO_{K}$ in $M_{\rho}$
and $\tau:\Gamma\to\GL n{\cO_{K}}$ a $\Gamma$-representation. Then
$\Gamma$ acts on $\cO_{M_{\rho}}^{\oplus n}$ in two ways. First
$\Gamma$ acts on $\cO_{M_{\rho}}$ and hence diagonally on $\cO_{M_{\rho}}^{\oplus n}$,
which determines 
\[
\delta:\Gamma\to\Aut_{\cO_{K}}(\cO_{M_{\rho}})^{\oplus n}.
\]
On the other hand, $\tau$ naturally extends to 
\[
\tilde{\tau}:\Gamma\to\GL n{\cO_{K}}\hookrightarrow\GL n{\cO_{M_{\rho}}}.
\]

\begin{defn}
We define the\emph{ tuning submodule} $\Xi_{\rho}\subset\cO_{M_{\rho}}^{\oplus n}$
to be the subset of elements on which the two actions $\delta$ and
$\tilde{\tau}$ coincide, that is, 
\[
\Xi_{\rho}:=\{x\in\cO_{M_{\rho}}^{\oplus n}\mid\forall\gamma\in\Gamma,\,\delta(\gamma)(x)=\tilde{\tau}(\gamma)(x)\}.
\]
When we want to specify $\tau$, we write it as $\Xi_{\rho}^{\tau}$. \end{defn}
\begin{lem}[\cite{Yasuda:2013fk}]
$\Xi_{\rho}$ is a free $\cO_{K}$-module of rank $n$.\end{lem}
\begin{proof}
This is obviously a torsion-free, hence free $\cO_{K}$-module. To
show that this has rank $n$, we take a trivial $\Gamma$-torsor $\Spec N\to\Spec L$
obtained from (\ref{eq:given torsor}) by a scalar extension $L/K$
. Then we write $N=\prod_{\gamma\in\Gamma}L_{\gamma}$ such that $L_{\gamma}$
are copies of $L$ and each $\gamma\in\Gamma$ induces an isomorphism
$L_{1}\xrightarrow{\sim}L_{\gamma}$ which is actually the identity
of $L$. Then $\Xi_{\rho}\otimes_{\cO_{K}}L\subset N^{\oplus n}$
is still the subset of those elements on which the two induced $\Gamma$-actions
coincide. Hence 
\[
\Xi_{\rho}\otimes_{\cO_{K}}L=\{(\tau(\gamma)(x))_{\gamma\in\Gamma}\mid x\in L^{\oplus n}\}\cong L^{\oplus n}.
\]
This shows that $\Xi_{\rho}$ has rank $n$.\end{proof}
\begin{defn}
We define a counting system $\bv_{\bullet}$ as follows: Let $x_{i}=(x_{ij})_{1\le j\le n}\in\cO_{M_{\rho}}^{\oplus n}$,
$1\le i\le n,$ be a basis of $\Xi_{\rho}$ as a free $\cO_{K}$-module.
Then we define $\bv_{K,\Gamma,\tau}(\rho)=\bv_{\tau}(\rho)\in\frac{1}{\sharp\Gamma}\ZZ_{\ge0}$
by
\[
\bv_{\tau}(\rho):=\frac{1}{\sharp\Gamma}\cdot\length\left(\frac{\cO_{M_{\rho}}}{(\det(x_{ij}))}\right)=\frac{1}{\sharp\Gamma}\cdot\length\left(\frac{\cO_{M_{\rho}}^{\oplus n}}{\cO_{M_{\rho}}\cdot\Xi_{\rho}}\right).
\]
\end{defn}
\begin{lem}
$\bv_{\bullet}$ is complete.\end{lem}
\begin{proof}
The additivity is obvious. To show the geometricity, it suffices to
show the identity of submodules of $\cO_{M_{\rho^{\nr}}}$, 
\[
\Xi_{\rho^{\nr}}=\Xi_{\rho}\otimes_{\cO_{M_{\rho}}}\cO_{M_{\rho^{\nr}}}.
\]
Indeed $\Xi_{\rho}$ is defined as the kernel of an $\cO_{K}$-linear
map
\[
(\delta(\gamma)-\tilde{\tau}(\gamma))_{\gamma\in\Gamma}:\cO_{M_{\rho}}^{\oplus n}\to(\cO_{M_{\rho}}^{\oplus n})^{\oplus\sharp\Gamma}
\]
and similarly for $\Xi_{\rho^{\nr}}$: denote them $\alpha$ and $\alpha_{\nr}$.
From the compatibilities of tensor product with completion and direct
limit, we have that $\cO_{M_{\rho}^{\nr}}=\cO_{M_{\rho}}\otimes_{\cO_{K}}\cO_{K^{\nr}}$.
We see that $\alpha_{\nr}$ is obtained from $\alpha$ by tensoring
$\cO_{K^{\nr}}$ over $\cO_{K}$, for instance, by looking at the
matrix representations of these maps. Therefore $\alpha_{\nr}$ is
obtained from $\alpha$ by tensoring $\cO_{M_{\rho}^{\nr}}$ over
$\cO_{M_{\rho}}$. The desired identity holds since $\cO_{M_{\rho^{\nr}}}$
is flat over $\cO_{M_{\rho}}$. 

To show the convertibility, it suffices to show it in the cases where
$\phi:\Gamma_{1}\to\Gamma_{2}$ in Definition \ref{def: complete}
is respectively injective and surjective. First we consider the former
case, supposing that $\Gamma_{1}$ is a subgroup of $\Gamma_{2}$.
Let $\rho_{1}\in S_{K,\Gamma_{1}}$, $\rho_{2}:=\phi\circ\rho_{1}\in S_{K,\Gamma_{2}}$
and 
\[
T_{i}=\Spec M_{i}\to\Spec K\quad(i=1,2)
\]
the corresponding $\Gamma_{i}$-torsors. Then $T_{2}$ is isomorphic
to the disjoint union of copies of $T_{1}$. Therefore if $\gamma_{1}=1,\gamma_{2},\dots,\gamma_{l}\in\Gamma_{2}$
are representatives of $\Gamma_{1}$-cosets, we can write 
\[
\cO_{M_{2}}=\prod_{i}\cO_{M_{1},i},
\]
with $\cO_{M_{1},i}$ copies of $\cO_{M_{1}}$. Moreover we suppose
that for each $i$, the $\gamma_{i}$-action on $\cO_{M_{2}}$ induces
the isomorphism $\cO_{M_{1},1}\to\cO_{M_{1},i}$ coming from the identity
of $\cO_{M_{1}}$. If $\Xi_{i}\subset\cO_{M_{i}}^{\oplus n}$, $i=1,2$,
are the tuning modules, then 
\begin{align*}
\Xi_{2} & =\{x\in\cO_{M_{2}}^{\oplus n}\mid\forall\gamma\in\Gamma_{2},\,\delta(\gamma)(x)=\tau(\gamma)(x)\}\\
 & =\{(x_{i})\in\prod_{i}\cO_{M_{1},i}^{\oplus n}\mid x_{1}\in\Xi_{1}\text{ and }x_{i}=\gamma_{i}(x_{1})\}.
\end{align*}
This shows the convertibility in this case. 

Next consider the case where $\phi$ is surjective. We may suppose
that $\Gamma_{2}$ is the quotient of $\Gamma_{1}$ by a normal subgroup
$N$. With the same notation as above, we have $T_{2}=T_{1}/N$ and
$\cO_{M_{2}}=\cO_{M_{1}}^{N}$. Then we see that $\Xi_{1}$ is contained
in $\cO_{M_{2}}^{\oplus n}$ and identical to $\Xi_{2}$. This finishes
the proof. 
\end{proof}
Let $\Gamma'\subset\Gamma$ be the stabilizer of a connected component
of $T_{\rho^{\nr}}=\Spec M_{\rho^{\nr}}$, that is, if $U$ is a connected
component, then $\Gamma'=\{\gamma\in\Gamma\mid\gamma(U)=U\}$. This
subgroup is unique up to conjugacy. Denoting the residue field of
$K$ by $\kappa,$ we consider the following $\Gamma'$-representation
over $\kappa$, 
\[
\Gamma'\hookrightarrow\Gamma\xrightarrow{\tau}\GL n{\cO_{K}}\to\GL n{\kappa},
\]
and write its fixed point locus as $(\kappa^{n})^{\Gamma'}.$ 
\begin{defn}
\label{def: residual}We define the \emph{residual tame part }by
\[
\bar{\bt}{}_{\tau}(\rho):=\codim\,(\kappa^{n})^{\Gamma'}.
\]

\end{defn}
The following is obvious from the definition.
\begin{lem}
The counting system $\bar{\bt}{}_{\bullet}$ is complete.\end{lem}
\begin{defn}
\label{def: weight}We define the\emph{ weight }of $\rho\in S_{K,\Gamma}$
with respect to $\tau$ as 
\[
\bw_{\tau}(\rho):=\bar{\bt}{}_{\tau}(\rho)-\bv_{\tau}(\rho)\in\frac{1}{\sharp\Gamma}\cdot\ZZ_{\le n}.
\]
\end{defn}
\begin{cor}
The counting system $\bw_{\bullet}$ is complete.\end{cor}
\begin{proof}
This follows from the completeness of $\bv_{\bullet}$ and $\bar{\bt}{}_{\bullet}$. 
\end{proof}

\section{A comparison of the Artin conductor and the weight}

For a representation $\tau:\Gamma\to\GL n{\cO_{K}}$, we define a
function $\ba_{\tau}$ on $S_{K,\Gamma}$ by $\ba_{\tau}:=\ba_{\iota\circ\tau}$
with $\iota:\GL n{\cO_{K}}\hookrightarrow\GL nK$. We will study a
relation among the Artin conductor $\ba_{\tau}$, the weight $\bw_{\tau}$
and their relatives.

\subsection{The tame case}

In this subsection, we consider the tame case, that is, $\sharp\Gamma$
is prime to the residue characteristic of $K$. Since the involved
counting systems are complete, in particular, geometric, without loss
of generality, we now suppose that $K$ has an algebraically closed
residue field. 

For $\rho\in S_{K,\Gamma},$ let $L_{\rho}/K$ be the Galois extension
defined by the kernel of $\tau\circ\rho$. Its Galois group $G:=\Gal(L_{\rho}/K)$
is cyclic, say having a generator $g$. As for higher ramification
groups, we have $G_{0}=G$ and $G_{i}=1$ for $i\ge1$. Therefore
we easily see the following.
\begin{lem}
We have
\[
\ba_{\tau}(\rho)=\bt_{\tau}(\rho)=\codim\,(K^{n})^{G}\text{ and }\bs_{\tau}(\rho)=0.
\]
\end{lem}
\begin{cor}
If $\tau^{\vee}$ is the dual representation of $\tau$, then
\[
\ba_{\tau}=\ba_{\tau^{\vee}}.
\]
\end{cor}
\begin{proof}
From the preceding lemma, $\ba_{\tau}(\rho)$ is equal to the number
of eigenvalues of the $g$-action on $K^{n}$ not equal to 1. This
number does not change by passing to the dual representation, hence
the corollary holds.
\end{proof}
For a positive integer $l$ and a uniformizer $\pi$ of $K$, we can
write $L_{\rho}=K(\sqrt[l]{\pi})$. 
\begin{lem}[{cf. \cite[Example 6.7]{Yasuda:2013fk}}]
\label{lem: v tame explicit}Let $\zeta\in K$ be the primitive $l$-th
root of unity such that $g(\sqrt[l]{\pi})=\zeta\cdot\sqrt[l]{\pi}$,
and let $\zeta^{a_{1}},\dots,\zeta^{a_{n}}$, $0\le a_{i}<l$, be
the eigenvalues of $g\in\GL nK$. Then 
\[
\bv_{\tau}(\rho)=\frac{1}{l}\cdot\sum_{i=0}^{n}a_{i}.
\]
\end{lem}
\begin{proof}
We first claim that $g$ is diagonalizable over $\cO_{K}$. Let $\xi$
be a power of $\zeta$ and $M_{\xi}\subset(\cO_{K})^{\oplus n}$ be
the $\xi$-eigenmodule of $g$. The inclusion map $M_{\xi}\hookrightarrow(\cO_{K})^{\oplus n}$
admits a splitting,
\[
\sigma_{\xi}:(\cO_{K})^{\oplus n}\to M_{\xi},\, x\mapsto\frac{1}{l}\sum_{i=0}^{l-1}g^{i}\xi^{-i}x.
\]
The product of $\sigma_{\zeta^{i}}$, $0\le i<l$ gives the inverse
of 
\[
\bigoplus_{i=0}^{l-1}M_{\zeta^{i}}\to(\cO_{K})^{\oplus n},\,(x_{i})\mapsto\sum x_{i}.
\]
This shows the claim. If we suppose $g=\diag(\zeta^{a_{1}},\dots,\zeta^{a_{n}})$,
we can choose 
\[
(\sqrt[l]{\pi}^{a_{1}},0,\dots,0),\dots,(0,\dots,0,\sqrt[l]{\pi}^{a_{n}})
\]
as a basis of the tuning submodule $\Xi_{\rho}\subset(\cO_{L_{\rho}})^{\oplus n}$.
This proves the lemma.
\end{proof}
Note that if we write the eigenvalues $\zeta^{a_{i}}$, $0\le a_{i}<l$
as $\zeta^{b_{i}}$, $0<b_{i}\le l$, then 
\[
\bw_{\tau}(\rho)=\sharp\{i\mid a_{i}\ne0\}-\frac{1}{l}\sum_{i=1}^{n}a_{i}=n-\frac{1}{l}\sum_{i=1}^{n}b_{i}.
\]

\begin{defn}
We say that the representation $\tau$ is \emph{balanced }if for every
$\gamma\in\Gamma,$ every non-real eigenvalue of $\tau(\gamma)$ appears
in a pair with its inverse, that is, $\tau(\gamma)$ have eigenvalues
\[
\epsilon_{1},\epsilon_{1}^{-1},\epsilon_{2},\epsilon_{2}^{-1},\dots,\epsilon_{m},\epsilon_{m}^{-1},1,\dots,1,-1,\dots,-1
\]
up to permutation.
\end{defn}
The following are examples of balanced representations:
\begin{enumerate}
\item a representation defined over a field which can be embedded in $\RR$.
\item a permutation representation.
\item a self-dual representation, that is, a representation which is isomorphic
to its dual. 
\item $\tau\oplus\tau^{\vee}$ for any $\tau$.\end{enumerate}
\begin{prop}
\label{prop:tame-compare}Suppose that $\tau$ is balanced. Then 
\[
\bw_{\tau}=\bv_{\tau}=\frac{1}{2}\cdot\bar{\bt}_{\tau}=\frac{1}{2}\cdot\bt_{\tau}=\frac{1}{2}\cdot\ba_{\tau}.
\]
\end{prop}
\begin{proof}
With the notation as above, the generator $g$ of $G=\Gal(L_{\rho}/K)$
has eigenvalues 
\[
\zeta^{b_{1}},\zeta^{l-b_{1}},\cdots,\zeta^{b_{m}},\zeta^{l-b_{m}},\overset{r}{\overbrace{1,\dots,1}},\overset{s}{\overbrace{-1,\dots,-1}}\quad(0<b_{i}<l).
\]
Then $\ba_{\tau}(\rho)=2m+s$, while 
\[
\bv_{\tau}(\rho)=\frac{1}{l}\sum_{i=1}^{m}(b_{i}+(l-b_{i}))+\frac{1}{l}\sum_{j=1}^{s}\frac{l}{2}=m+\frac{s}{2}.
\]
We have proved that $\bv_{\tau}=\ba_{\tau}/2$.

Since nontrivial eigenvalues of $g\in\GL n{\cO_{K}}\subset\GL nK$
stay nontrivial by passing to the residue field $\kappa$, we have
\[
\bar{\bt}{}_{\tau}(\rho)=\codim\,(\kappa^{n})^{g}=\codim\,(K^{n})^{g}=\bt_{\tau}(\rho).
\]
This proves the remaining equalities.\end{proof}
\begin{cor}
\label{cor:compare tame w a}Provided that $\sharp\Gamma$ is prime
to the residue characteristic, we have that for any $\tau$, 
\[
\bw_{\tau\oplus\tau^{\vee}}=\ba_{\tau}.
\]
\end{cor}
\begin{proof}
Since $\tau\oplus\tau^{\vee}$ is balanced,
\[
\bw_{\tau\oplus\tau^{\vee}}=\frac{1}{2}\cdot\ba_{\tau\oplus\tau^{\vee}}.
\]
Since $\ba_{\bullet}$ is additive and $\ba_{\tau}=\ba_{\tau^{\vee}}$,
the corollary follows.\end{proof}
\begin{rem}
Without the tameness condition, $\bt_{\bullet}$ and $\bar{\bt}_{\bullet}$
do not generally coincide. For instance, if $K=\kappa((t))$ with
$\kappa$ a perfect field of characteristic $p>0$, then for the representation
\[
\tau:\ZZ/p\xrightarrow{\sim}\left\langle \begin{pmatrix}1 & t\\
0 & 1
\end{pmatrix}\right\rangle \subset\GL 2{\kappa[[t]]}
\]
and $\rho\in S_{K,\ZZ/p}$ defining a ramified $\ZZ/p$-extension
of $K$, we have 
\[
\bt_{\tau}(\rho)=1\text{ and }\bar{\bt}_{\tau}(\rho)=0.
\]

\end{rem}

\subsection{Permutation representations}

We now drop the requirement that $\sharp\Gamma$ is prime to the residue
characteristic of $K$, and suppose instead that $\tau:\Gamma\to\GL n{\cO_{K}}$
is a permutation representation. If it is tame, then it is balanced
and Proposition \ref{prop:tame-compare} applies. Otherwise, the assertion
of the proposition fails. However some equalities are still valid. 
\begin{thm}
\label{thm:permutation}For a permutation representation $\tau$,
we have
\[
\bar{\bt}_{\tau}=\bt_{\tau}\text{ and }\bv_{\tau}=\frac{1}{2}\cdot\ba_{\tau}.
\]
\end{thm}
\begin{proof}
The first equality is obvious. For the second, let $\rho\in S_{K,\Gamma}.$
In order to compute both sides of the second equality, since $\ba_{\bullet}$
and $\bv_{\bullet}$ are complete, we may reduce to the case where
$\tau$ is an inclusion map, $\Gamma$ transitively acts on $\{1,\dots,n\}$
and $\rho$ is surjective. Let $T_{\rho}=\Spec M_{\rho}\to\Spec K$
be the $\Gamma$-torsor corresponding to $\rho$. From the $\Gamma$-action
on $\{1,\dots,n\},$ $\rho$ defines also an \'etale extension $E_{\rho}/K$
of degree $n$. From the assumption we have just made, $M_{\rho}$
and $E_{\rho}$ are fields. Let $\Delta\subset\Gamma$ be the stabilizer
of $1\in\{1,\dots,n\}$. Then we can identify $E_{\rho}$ with the
$\Delta$-invariant subfield $M_{\rho}^{\Delta}\subset M_{\rho}$.
Let $\sigma_{i}\in\Delta$, $i=1,\dots,n,$ be such that $\sigma_{i}(1)=i.$
Then, by definition,
\[
\Xi_{\rho}=\{(\sigma_{1}(e),\dots,\sigma_{n}(e))\in\cO_{M_{\rho}}^{\oplus n}\mid e\in\cO_{E_{\rho}}\}.
\]
Note that $\sigma_{i}(e)$ is independent of the choice of $\sigma_{i}$.
If $e_{1},\dots,e_{n}$ are a basis of $\cO_{E_{\rho}}$ as an $\cO_{K}$-module,
then $ $$(\sigma_{1}(e_{i}),\dots,\sigma_{n}(e_{i})),$ $i=1,\dots,n$,
are a basis of $\Xi_{\rho}$. Since 
\[
d_{E_{\rho}/K}=\det(\sigma_{j}(e_{i}))^{2},
\]
we have 
\[
\bv_{\tau}(\rho)=\frac{1}{n!}\cdot v_{M_{\rho}}(\det(\sigma_{j}(e_{i})))=\frac{1}{2}\cdot v_{K}(d_{E_{\rho}/K}).
\]
Lemma \ref{lem:Kedlaya} completes the proof. \end{proof}
\begin{cor}
\label{cor:comparison twice}For a permutation representation $\tau$,
let $2\tau:=\tau\oplus\tau.$ Then
\[
\bw_{2\tau}=2\cdot\bt_{\tau}-\ba_{\tau}=\bt_{\tau}-\bs_{\tau}.
\]

\end{cor}
For permutation representations, the expression $\bw_{2\tau}=\bt_{\tau}-\bs_{\tau}$
shows that the tame part contributes positively to the weight, while
the $p$-part contributes negatively, as observed in \cite{Yasuda:2013fk}
in some other cases. It is in contrast with that both the tame and
$p$- parts contribute positively to the Artin conductor: $\ba_{\tau}=\bt_{\tau}+\bs_{\tau}$.
\begin{example}
Theorem \ref{thm:permutation} and Corollary \ref{cor:comparison twice}
fail without the assumption that $\tau$ is a permutation representation.
Suppose that $K$ has characteristic $p>0$ and $\Gamma=\left\langle \gamma\right\rangle \cong\ZZ/(p).$
For $n\le p$, consider the representation $\tau:\Gamma\to\GL n{\cO_{K}}$
with 
\[
\tau(\gamma)=\begin{pmatrix}1 & 1\\
 & 1 & 1\\
 &  & \ddots & \ddots\\
 &  &  & 1 & 1\\
 &  &  &  & 1
\end{pmatrix}.
\]
For $\rho\in S_{K,\Gamma}$ with $L_{\rho}/K$ ramified, let $j$
be the largest integer $i$ with $G_{i}\ne1.$ By an easy calculation,
we have 
\begin{gather*}
\ba_{\tau}(\rho)=(j+1)(n-1)\text{ and}\\
\bt_{\tau}(\rho)=\bar{\bt}_{\tau}(\rho)=n-1.
\end{gather*}
On the other hand, from \cite[Example 6.8]{Yasuda:2013fk}, we have
\[
\bw_{\tau}(\rho)=-\sum_{a=1}^{n-1}\left\lfloor \frac{ja}{p}\right\rfloor .
\]
For instance, if $n=2$, then 
\[
2\bt_{\tau}-\ba_{\tau}=1-j,
\]
while 
\[
\bw_{2\tau}=-2\left\lfloor \frac{j}{p}\right\rfloor .
\]
They are not equal unless $p=2$.
\end{example}

\section{The McKay correspondence}

In this section, we discuss several versions of the McKay correspondence
in terms of motivic invariants. We also formulate \emph{point counting
realizations} of these correspondences and prove them in some cases.

\subsection{The Grothendieck ring of varieties and the point counting realization}

We first describe the set to which motivic invariants belong. Let
$\kappa$ be a perfect field. The Grothendieck ring of $\kappa$-varieties,
denoted $K_{0}(\Var_{\kappa})$, is the free abelian group of isomorphic
classes $[V]$ of $\kappa$-varieties modulo the following relation:
if $W$ is a closed subvariety of $V$, then $[V]=[W]+[V\setminus W]$.
This indeed becomes a ring by the multiplication defined by $[V]\cdot[V']:=[V\times_{\kappa}V']$.
For a constructible subset $C\subset V$, its class $[C]$ in $K_{0}(\Var_{\kappa})$
is defined in a natural way. We define a distinguished element $\LL:=[\AA_{\kappa}^{1}]$
of $K_{0}(\Var_{\kappa})$. If $\kappa=\FF_{q}$, then there exists
a unique ring map
\[
\sharp:K_{0}(\Var_{\kappa})\to\ZZ
\]
sending $[V]$ to $\sharp V(\kappa)$, which is called the \emph{point
counting realization.}

Motivic integrals which we will consider below take values in a certain
modification of $K_{0}(\Var_{\kappa})$, the semiring $\mathfrak{R}^{1/r}$
in \cite[Section 3.8]{MR2271984}, where $r$ is the order of the
given finite group $\Gamma$. Alternatively, we may use the complete
Grothendieck ring of mixed $G_{\kappa}$-representations over $\QQ_{l}$,
$\hat{K}_{0}(MR(G_{\kappa},\QQ_{l}))$, with $l$ a prime number prime
to $r$ and to the characteristic of $\kappa$. As in $K_{0}(\Var_{\kappa})$,
there exists an element $[V]\in\mathfrak{R}^{1/r}$ associated to
each $\kappa$-variety $V$. If $W$ is a closed subvariety of $V$,
then we again have $[V]=[W]+[V\setminus W]$. The product of elements
$[V]$ and $[V']$ is similarly given. Moreover $\mathfrak{R}^{1/r}$
contains fractional powers $\LL^{i}$ of $\LL=[\AA_{\kappa}^{1}]$
for $i\in\frac{1}{r}\ZZ$. For $i\in\frac{1}{r}\ZZ$ with $i>0$,
the infinite sum 

\[
1+\LL^{-i}+\LL^{-2i}+\cdots
\]
converges in $\mathfrak{R}^{1/r}$. $ $Let $\mathfrak{R}_{0}^{1/r}$
be the subsemiring of $\mathfrak{R}^{1/r}$ generated by the classes
$[V]$ of $\kappa$-varieties, fractional powers $\LL^{i}$, $i\in\frac{1}{r}\ZZ$
of $\LL$ and the infinite sums of the above form. We then have the
point counting realization, 
\[
\sharp:\mathfrak{R}_{0}^{1/r}\to\RR,
\]
which sends $[V]$ to $\sharp V(\kappa)$, $\LL^{-i}$ to $q^{-i}$,
and the infinite sum $1+\LL^{-i}+\LL^{-2i}+\cdots$ to 
\[
1+q^{-i}+q^{-2i}+\cdots=\frac{1}{1-q^{-i}}.
\]
We can check that this map is actually well defined for instance by
using the realization map $\mathfrak{R}^{1/r}\to\hat{K}_{0}(MR(G_{\kappa},\QQ_{l}))$
given in \cite[Section 3.8]{MR2271984} (see also \cite[page 1142]{MR3230848}).

\subsection{The tame case}

The classical McKay correspondence is about ADE surface singularities
over $\CC$ and originates in McKay's observation in the late 70 that
the same Dynkin diagram appears both in the minimal resolution of
singularities and in representation theory. Later, it has been refined
and generalized in diverse directions. The first hint of a generalization
to higher dimensions came from physics \cite{MR851703,MR818423}.
For more details on the McKay correspondence in general, we refer
the reader to a nice survey paper \cite{MR1886756}. 

A version of the McKay correspondence in terms of motivic invariants
was obtained by Batyrev \cite{MR1677693}, where the base field was
$\CC$. His result was a considerable refinement of the \textquotedblleft{}Physicists\textquoteright{}
Euler number conjecture\textquotedblright{}. A slightly different
version was obtained by Denef-Loeser \cite{MR1905024} over a field
of characteristic zero containing all $\sharp\Gamma$-th roots of
unity with $\Gamma$ the given finite group. Yasuda \cite{MR2271984}
then generalized it to an arbitrary perfect field of characteristic
prime to $\sharp\Gamma$. Now we recall these results. For the sake
of intelligibility, we first see the case of a perfect field (possibly
of positive characteristic) containing all $\sharp\Gamma$-th roots
of unity. 

Let $\kappa$ be a perfect field and $\tau:\Gamma\hookrightarrow\GL n{\kappa}$
a faithful representation of a finite group $\Gamma$, through which
we regard $\Gamma$ as a subgroup of $\GL n{\kappa}$. The group $\Gamma$
acts on the polynomial ring $\kappa[x_{1},\dots,x_{n}]$ by naturally
extending the $\Gamma$-action on the linear part $\bigoplus_{i}\kappa\cdot x_{i}\cong\kappa^{n}$.
Let $\AA_{\kappa}^{n}:=\Spec\kappa[x_{1},\dots,x_{n}]$ and $X:=\AA_{\kappa}^{n}/\Gamma$
the quotient scheme. Since $X$ is a $\QQ$-Gorenstein variety over
$\kappa$, for a resolution of singularities $f:Y\to X,$ we can define
the relative canonical divisor $K_{Y/K}:=K_{Y}-f^{*}K_{X}$, which
is a $\QQ$-divisor with support contained in the exceptional locus.
We say that $f$ is \emph{crepant }if $K_{Y/K}=0.$ We let $0\in X(\kappa)$
be the image of the origin of $\AA_{\kappa}^{n}$. 

We suppose that $\kappa$ contains all $\sharp\Gamma$-th roots of
unity and choose a primitive $\sharp\Gamma$-th root $\zeta$. 
\begin{defn}
\label{def: weight matrix}For $g\in\Gamma\subset\GL n{\kappa}$,
if it is equivalent to the diagonal matrix $\diag(\zeta^{b_{1}},\dots,\zeta^{b_{n}})$
with $1\le b_{i}\le\sharp\Gamma$, we put 
\[
\bw_{\tau}(g):=n-\frac{1}{\sharp\Gamma}\sum_{i=1}^{n}b_{i}\in\QQ.
\]

\end{defn}
The rational number $\bw_{\tau}(g)$ depends only on $\tau$ and the
conjugacy class of $g$. We let $\Conj{\Gamma}$ be the set of conjugacy
classes of $\Gamma$. The following is the McKay correspondence in
the simplest case:
\begin{thm}[\cite{MR1905024,MR2271984}]
\label{thm: McKay tame 1}Suppose that $\Gamma\subset\GL n{\kappa}$
has no pseudo-reflection. Namely, for any $g\in\Gamma\setminus\{1\}$,
the fixed point locus $(\kappa^{n})^{g}$ has codimension at least
two. Suppose that $\sharp\Gamma$ is prime to the characteristic of
$\kappa$ and that $\kappa$ contains all $\sharp\Gamma$-th roots
of unity. Then, for a crepant resolution $f:Y\to X$, we have
\[
[f^{-1}(0)]=\sum_{[g]\in\Conj{\Gamma}}\LL^{\bw_{\tau}(g)}\in\mathfrak{R}_{0}^{1/\sharp\Gamma}.
\]

\end{thm}
When $\kappa$ has characteristic zero and $\Gamma$ is contained
in $\mathrm{SL}_{n}(\kappa)$, the theorem was proved in \cite{MR1905024}.
To deduce the general case of the theorem from a result in \cite{MR2271984},
we need to first prove Theorem \ref{thm: McKay tame 2} below. 
\begin{rem}
The weight $\bw_{\tau}(g)$ defined above in the tame case is essentially
the same thing as what is called \emph{age \cite{MR1463181} }or \emph{fermion
number shift \cite{MR1233848}.}
\end{rem}
To generalize the result to the tame case over an arbitrary perfect
base field, we need the notion of \emph{twisted 0-jets} from the paper
\cite{MR2271984}, and in particular, need to use Deligne-Mumford
stacks. 

Let $l$ be a positive integer prime to the characteristic of $\kappa$.
For a Deligne-Mumford stack $\cZ$ over $\kappa$, the \emph{stack
of twisted $0$-jets of order $l$, $\cJ_{0}^{l}\cZ$, }is the Deligne-Mumford
stack whose points are given as follows. Let $\mu_{l,\kappa}:=\Spec\kappa[t]/(t^{l}-1)$
be the group scheme of $l$-th roots of unity defined over $\kappa$.
For a $\kappa$-algebra $R$, an $R$-point of $\cJ_{0}^{l}\cZ$ is
a representable $\kappa$-morphism of the form
\[
[\Spec R/\mu_{l,\kappa}]\to\cZ,
\]
where $\mu_{l,\kappa}$ acts on $\Spec R$ trivially and $[\Spec R/\mu_{l,\kappa}]$
is the associated quotient stack. The natural morphism $\Spec R\to[\Spec R/\mu_{l,\kappa}]$
gives a natural morphism $\cJ_{0}^{l}\cZ\to\cZ$. 
\begin{rem}
Twisted 0-jets are a special case of twisted $n$-jets given by representable
morphisms
\[
[(\Spec R[t]/(t^{nl+1}))/\mu_{l,\kappa}]\to\cZ.
\]
Twisted jets and twisted arcs, which are representable morphisms 
\[
[\Spec R[[t]]/\mu_{l,\kappa}]\to\cZ,
\]
naturally appeared in a generalization of motivic integration to Deligne-Mumford
stacks.
\end{rem}
Let $\cX:=[\AA_{\kappa}^{n}/\Gamma]$ be the quotient stack associated
to the $\Gamma$-action on $\AA_{\kappa}^{n}$. The natural morphism
$\AA_{\kappa}^{n}\to X$ factors through $\cX$ and the induced morphism
$\cX\to X$ is finite and birational. We put 
\[
\cJ_{0}\cX:=\bigsqcup_{p\nmid l}\cJ_{0}^{l}\cX.
\]
Since $\cJ_{0}^{l}\cX$ is empty for $l\nmid\sharp\Gamma$, the disjoint
union is a finite union. 
\begin{rem}
The stack $\cJ_{0}\cX$ is a twisted form of the inertia stack $I\cX:=\cX\times_{\Delta,\cX\times\cX,\Delta}\cX$
with $\Delta:\cX\to\cX\times\cX$ the diagonal morphism. Namely $\cJ_{0}\cX$
and $I\cX$ become isomorphic after the base change from $\kappa$
to an algebra closure $\bar{\kappa}$.
\end{rem}
We put $(\cJ_{0}\cX)_{0}$ to be the fiber product $(\cJ_{0}\cX)\times_{X}\Spec\kappa$
with respect to the composition morphism $\cJ_{0}\cX\to\cX\to X$
and the point $0\in X(\kappa)$. The stack $(\cJ_{0}\cX)_{0}$ is
a closed substack of $\cJ_{0}\cX$. We denote its coarse moduli space
by $\overline{(\cJ_{0}\cX)_{0}}$, which is a scheme finite over $\kappa$.
For an algebraic closure $\bar{\kappa}$ of $\kappa$, we have an
identification 
\[
\overline{(\cJ_{0}\cX)_{0}}(\bar{\kappa})=(\cJ_{0}\cX)_{0}(\bar{\kappa})/\cong.
\]
Here the right hand side is the set of isomorphism classes of morphisms
$\Spec\bar{\kappa}\to(\cJ_{0}\cX)_{0}.$ 

Identifying $\Spec\bar{\kappa}$ with the origin of $\AA_{\bar{\kappa}}^{n}$,
we can identify $[\Spec\bar{\kappa}/\Gamma]$ with a closed substack
of $\cX\otimes_{\kappa}\bar{\kappa}$. Then $\overline{(\cJ_{0}\cX)_{0}}(\bar{\kappa})$
is identified with the isomorphism classes of representable morphisms
over $\kappa$
\[
[\Spec\bar{\kappa}/\mu_{l}]\to[\Spec\bar{\kappa}/\Gamma]
\]
for $l\mid\sharp\Gamma$. Here $\mu_{l}$ is the group of $l$th roots
of unity in $\bar{\kappa}$ and later identified with the group scheme
$\mu_{l,\bar{\kappa}}$. Since the automorphism groups of the $\bar{\kappa}$-points
of these quotient stacks are $\mu_{l}$ and $\Gamma$ respectively,
the morphism gives a homomorphism 
\[
\mu_{l}\to\Gamma,
\]
which is injective because the stack morphism is representable. Now
two representable morphisms $[\Spec\bar{\kappa}/\mu_{l}]\rightrightarrows[\Spec\bar{\kappa}/\Gamma]$
are isomorphic if and only if the corresponding maps $\mu_{l}\rightrightarrows\Gamma$
are conjugate to each other in $\Gamma$. Let $\Conj{\mu_{l},\Gamma}$
be the set of injective homomorphisms modulo conjugation in $\Gamma$.
We thus have a natural one-to-one correspondence
\[
\overline{(\cJ_{0}\cX)_{0}}(\bar{\kappa})\leftrightarrow\bigsqcup_{l\mid\sharp\Gamma}\Conj{\mu_{l},\Gamma}.
\]

We fix a primitive $\sharp\Gamma$-th root $\zeta\in\bar{\kappa}$
so that we obtain canonical choices of generators of $\mu_{l}$ for
all $l\mid\sharp\Gamma$, that is, $\zeta_{l}:=\zeta^{\sharp\Gamma/l}$.
Then we can identify an injective map $\mu_{l}\to\Gamma$ with the
image of $\zeta_{l}$ in $\Gamma$, and $\bigsqcup_{l\mid\sharp\Gamma}\Conj{\mu_{l},\Gamma}$
with $\Conj{\Gamma}$. Thus we obtain a one-to-one correspondence
\[
\overline{(\cJ_{0}\cX)_{0}}(\bar{\kappa})\leftrightarrow\Conj{\Gamma}.
\]

For later use, we would like to know what subset of $\Conj{\Gamma}$
corresponds to the subset of $\kappa$-points, $\overline{(\cJ_{0}\cX)_{0}}(\kappa)\subset\overline{(\cJ_{0}\cX)_{0}}(\bar{\kappa})$,
when $\kappa$ is finite. If $q$ is the cardinality of a finite field
$\kappa$, then we put 
\[
\cC_{\Gamma,q}:=\{[g]\in\Conj{\Gamma}\mid[g]=[g^{q}]\}\subset\Conj{\Gamma}.
\]

\begin{lem}
Suppose $\kappa=\FF_{q}$. By the correspondence $\overline{(\cJ_{0}\cX)_{0}}(\bar{\kappa})\leftrightarrow\Conj{\Gamma}$,
the subset $\overline{(\cJ_{0}\cX)_{0}}(\kappa)$ corresponds to \textup{$\cC_{\Gamma,q}$. }\end{lem}
\begin{proof}
For a scheme $U$ over $\FF_{q}$, we denote by $F_{U}$ the Frobenius
morphism $U\to U$ defined by the sheaf morphism $\cO_{U}\to\cO_{U}$,
$f\mapsto f^{q}$. We define a self-map $F$ of $\overline{(\cJ_{0}\cX)_{0}}(\bar{\kappa})$
by 
\[
F(x):=x\circ F_{\Spec\bar{\kappa}}=F_{\overline{(\cJ_{0}\cX)_{0}}}\circ x.
\]
The set of $\kappa$-points, $\overline{(\cJ_{0}\cX)_{0}}(\kappa)$,
is the fixed point locus $\overline{(\cJ_{0}\cX)_{0}}(\bar{\kappa})^{F}$
of $F$. It suffices to show that the self-map $F$ of $\overline{(\cJ_{0}\cX)_{0}}(\bar{\kappa})$
corresponds to the self-map of $\cC_{\Gamma,q}$ given by $[g]\mapsto[g^{q}]$.
Note that Deligne-Mumford stacks over $\kappa$ also have Frobenius
morphisms, since for any \'{e}tale morphism $V\to U$ of $\kappa$-schemes,
$F_{V}$ is the base change of $F_{U}$. Let $x\in\overline{(\cJ_{0}\cX)_{0}}(\bar{\kappa})$
be a point corresponding to a morphism 
\[
\alpha:[\Spec\bar{\kappa}/\mu_{l}]\to[\Spec\bar{\kappa}/\Gamma]
\]
and to a homomorphism
\[
\beta:\mu_{l}\to\Gamma.
\]
Then $F(x)$ corresponds to the composition of $\alpha$ and the Frobenius
morphism of $[\Spec\bar{\kappa}/\mu_{l}]$, and to the composite map
of groups 
\[
\mu_{l}\xrightarrow{F_{\mu_{l}}}\mu_{l}\xrightarrow{\beta}\Gamma,
\]
where $F_{\mu_{l}}$ is the Frobenius morphism of the group scheme
$\mu_{l}=\mu_{l,\bar{\kappa}}$. Since $F_{\mu_{l}}$ sends $\xi\in\mu_{l}$
to $\xi^{q}$, the lemma follows.\end{proof}
\begin{cor}
If $\kappa$ contains all $\sharp\Gamma$-th roots of unity, then
$\overline{(\cJ_{0}\cX)_{0}}(\kappa)=\overline{(\cJ_{0}\cX)_{0}}(\bar{\kappa})$
and we have a one-to-one correspondence
\[
\{\text{connected components of }\overline{(\cJ_{0}\cX)_{0}}\}\leftrightarrow\overline{(\cJ_{0}\cX)_{0}}(\kappa)\leftrightarrow\Conj{\Gamma}.
\]
\end{cor}
\begin{proof}
Let $q$ be the cardinality of $\kappa$. The condition that $\kappa$
contains all $\sharp\Gamma$-th roots of unity implies that $\sharp\Gamma$
divides $q-1$. Therefore, for any $g\in\Gamma$, we have $g^{q-1}=1$.
It follows that the self-map $[g]\mapsto[g^{q}]$ of $\Conj{\Gamma}$
is the identity map. This shows the correspondence $\overline{(\cJ_{0}\cX)_{0}}(\kappa)\leftrightarrow\Conj{\Gamma}$
and the equality $\overline{(\cJ_{0}\cX)_{0}}(\kappa)=\overline{(\cJ_{0}\cX)_{0}}(\bar{\kappa})$.
Hence the associated reduced scheme of $\overline{(\cJ_{0}\cX)_{0}}(\kappa)$
is the union of $\sharp\Conj{\Gamma}$ copies of $\Spec\kappa$, which
proves the other correspondence.
\end{proof}
We define a function $\bw_{\tau}$ on $\overline{(\cJ_{0}\cX)_{0}}(\bar{\kappa})$
as the function corresponding to $\bw_{\tau}$ on $\Conj{\Gamma}$
through the correspondence above. 
\begin{rem}
\label{rem: sht}The resulting map $\bw_{\tau}:\overline{(\cJ_{0}\cX)_{0}}(\bar{\kappa})\to\QQ$
is the same as a restriction of the function $\mathrm{sht}:|\cJ_{0}\cX|\to\QQ$
defined in \cite[Section 3.9]{MR2271984} up to an involution of $\overline{(\cJ_{0}\cX)_{0}}(\bar{\kappa})$.
The involution corresponds the involution $[g]\leftrightarrow[g^{-1}]$
of $\Conj{\Gamma}$, and originates in the fact that tangent spaces
are considered in the cited paper, while we are implicitly considering
cotangent spaces. However the involution does not affect the counting
problem, and can be safely ignored.
\end{rem}
The value $\bw_{\tau}$ of this function actually depends only on
the connected component of $\overline{(\cJ_{0}\cX)_{0}}$ to which
the given point belongs. Therefore we write it as $\bw_{\tau}(C)$
for the connected component $C$.
\begin{thm}[\cite{MR2271984}]
\label{thm: McKay tame 2}Let $\kappa$ be an arbitrary perfect field.
Suppose that $\sharp\Gamma$ is prime to the characteristic of $\kappa$
and that $\Gamma\subset\GL n{\kappa}$ has no pseudo-reflection. For
a crepant resolution $f:Y\to X$, we have
\[
[f^{-1}(0)]=\sum_{C\subset\overline{(\cJ_{0}\cX)_{0}}}[C]\cdot\LL^{\bw_{\tau}(C)}\in\mathfrak{R}_{0}^{1/\sharp\Gamma}.
\]
\end{thm}
\begin{proof}
For a positive integer $r$ with $rK_{X}$ Cartier (for instance,
$r=\sharp\Gamma$), we define an ideal sheaf $\cG\subset\cO_{X}$
by 
\[
\mathrm{Im}((\Omega_{X}^{d})^{\otimes r}\to\cO_{X}(rK_{X}))=\cG\cdot\cO_{X}(rK_{X}).
\]
Let $J_{\infty}X$ be the arc space of $X$, parameterizing morphisms
$\Spec\kappa[[t]]\to X$, and $\pi_{0}:J_{\infty}X\to X$ the projection.
We define an invariant by a motivic integral as follows: 
\[
M:=\int_{\pi_{0}^{-1}(0)}\LL^{\frac{1}{r}\mathrm{ord}\,\cG}\, d\mu_{X}.
\]
Let $h$ be either $f$ or the morphism $\cX\to X$. Since $h$ is
a crepant resolution (in the category of Deligne-Mumford stacks),
a standard argument (see the proof of \cite[Proposition 90]{MR2271984})
shows that 
\[
\left(\frac{1}{r}\mathrm{ord}\,\cG\right)\circ h_{\infty}=\mathrm{ord}\,\mathrm{Jac}_{h},
\]
where $h_{\infty}$ is the map of (twisted) arc spaces associated
to $h$ and $\Jac_{h}$ is the Jacobian ideal sheaf of $h$ on $Y$
or $\cX$. Applying the transformation rule \cite[Theorem 66]{MR2271984}
to $f$, we get
\begin{align*}
M & =\int{}_{\pi_{0}^{-1}(f^{-1}(0))}\LL^{\left(\frac{1}{r}\mathrm{ord}\,\cG\right)\circ f_{\infty}-\mathrm{ord}\,\mathrm{Jac}_{f}}\, d\mu_{Y}\\
 & =\int{}_{\pi_{0}^{-1}(f^{-1}(0))}\, d\mu_{Y}=[f^{-1}(0)].
\end{align*}
Let $\pi$ be the projection $|\cJ_{\infty}\cX|\to|\cX|$ from the
space of twisted arcs of $\cX$ to the point set of $\cX$ and $\fs_{\cX}$
is the composition $|\cJ_{\infty}\cX|\to|\cJ_{0}\cX|\xrightarrow{\mathrm{sht}}\QQ$
(see Remark \ref{rem: sht}). Applying the transformation rule to
the morphism $\cX\to X$, which we denote by $p$, we get 
\begin{align*}
M & =\int{}_{\pi^{-1}(0)}\LL^{\left(\frac{1}{r}\mathrm{ord}\,\cG\right)\circ p_{\infty}-\mathrm{ord}\,\mathrm{Jac}_{p}+\fs_{\cX}}\, d\mu_{\cX}\\
 & =\int_{\pi^{-1}(0)}\LL^{\fs_{\cX}}\, d\mu_{\cX}=\sum_{C\subset\overline{(\cJ_{0}\cX)_{0}}}[C]\cdot\LL^{\bw_{\tau}(C)}.
\end{align*}
We have proved the theorem.\end{proof}
\begin{rem}
If the fiber product $Y\times_{X}\cX$ has a resolution of singularities,
then the theorem follows from \cite[Corollary 72]{MR2271984}. However,
since we are working in arbitrary characteristic and do not know the
existence of resolution, we avoided to use it. 
\end{rem}
When $\kappa$ contain all $\sharp\Gamma$-th roots of unity, then
$\overline{(\cJ_{0}\cX)_{0}}$ is the disjoint union of $\sharp\Conj{\Gamma}$
copies of $\Spec\kappa$, and we obtain Theorem \ref{thm: McKay tame 1}.
Applying the point counting realization $\sharp:\mathfrak{R}_{0}^{1/\sharp\Gamma}\to\RR$,
we obtain:
\begin{cor}
\label{cor: counting tame 1}Let $\kappa$ be a finite field. Suppose
that $\sharp\Gamma$ is prime to the characteristic of $\kappa$ and
that $\Gamma\subset\GL n{\kappa}$ has no pseudo-reflection. For a
crepant resolution $f:Y\to X$, we have
\[
\sharp f^{-1}(0)(\kappa)=\sum_{[g]\in\cC_{\Gamma,q}}q^{\bw_{\tau}(g)}.
\]

\end{cor}

\subsection{The wild case}

A conjectural generalization of the McKay correspondence to the wild
case was formulated by Yasuda \cite{Yasuda:2013fk}, after studying
the special case of a cyclic group of prime order \cite{MR3230848}.
Actually it was a generalization not only to the wild case but also
the relative setting over the integer ring of a local field. Let $K$
be a local field and $\tau:\Gamma\to\GL n{\cO_{K}}$ a faithful representation
of a finite group $\Gamma$. Note that a representation $\Gamma\to\GL n{\kappa}$
as in the last subsection is considered as a special case by composing
it with the inclusion map $\GL n{\kappa}\hookrightarrow\GL n{\kappa[[t]]}$.
We can similarly construct the quotient scheme $X:=\AA_{\cO_{K}}^{n}/\Gamma$.
We can define the canonical divisor $K_{X}$ over $\cO_{K}$ (see
\cite[page 8]{MR3057950}), which is $\QQ$-Cartier. For a proper
birational morphism $f:Y\to X$ with $Y$ smooth over $\cO_{K}$,
we can say that $f$ is a \emph{crepant resolution} if $K_{Y}=f^{*}K_{X}$. 

We need a conjectural moduli space of (not pointed) \'etale $\Gamma$-torsors
of $\Spec K$: we denote it by $\Gamma\text{-}\mathrm{Cov}(K)$. The
base field of the space should be not $K$ but $\kappa$. In the tame
case, $\Gamma\text{-}\mathrm{Cov}(K)$ is expected to be zero dimensional.
Moreover, if $\kappa$ is algebraically closed, then it should consist
of $\sharp\Conj{\Gamma}$ points. On the other hand, in the wild case,
the space is expected to be infinite dimensional. A (not pointed)
$\Gamma$-torsor over $\Spec K$ corresponds to a continuous homomorphism
$G_{K}\to\Gamma$ modulo conjugation in $\Gamma$. Through this correspondence,
we can define a function 
\[
\bw_{\tau}:\Gamma\text{-}\mathrm{Cov}(K)\to\QQ
\]
corresponding to the function $\bw_{\tau}$ on $S_{K,\Gamma}$. We
expect that this function has finite dimensional constructible subsets
as fibers (at least when $X$ admits a crepant resolution) and expressions
$[\bw_{\tau}^{-1}(r)]$, $r\in\QQ$ make sense as elements of $\mathfrak{R}^{1/\sharp\Gamma}$.
The motivic integral of $\LL^{\bw_{\tau}}:\Gamma\text{-}\mathrm{Cov}(K)\to\mathfrak{R}^{1/r}$
is then defined as
\[
\int_{\Gamma\text{-}\mathrm{Cov}(K)}\LL^{\bw_{\tau}}:=\sum_{r\in\QQ}[\bw_{\tau}^{-1}(r)]\LL^{r}\in\mathfrak{R}^{1/\sharp\Gamma}\cup\{\infty\}.
\]

The following conjecture was formulated in \cite{Yasuda:2013fk},
under the assumption that the residue field $\kappa$ is algebraically
closed. 
\begin{conjecture}
\label{conj: McKay motivic}Suppose that for every $g\in\Gamma\setminus\{1\}$,
the fixed point locus $(\AA_{\cO_{K}}^{n})^{g}$ has codimension at
least two, equivalently that the quotient map $\AA_{\cO_{K}}^{n}\to X$
is \'etale in codimension one. Let $f:Y\to X$ be a crepant resolution
and $Z\subset Y$ be the preimage of the origin $0\in X(\kappa)$.
Then 
\[
[Z]=\int_{\Gamma\text{-}\mathrm{Cov}(K)}\LL^{\bw_{\tau}}\in\mathfrak{R}_{0}^{1/\sharp\Gamma}.
\]
\end{conjecture}
\begin{rem}
One difficulty in formulating the conjecture was to find out the correct
function $\bw_{\tau}$. In \cite{Yasuda:2013fk}, the function was
defined only when $\kappa$ is algebraically closed. Studying a relation
with the Artin conductor, we realized that we need to use $T_{\rho^{\nr}}$
rather than $T_{\rho}$ for Definition \ref{def: residual}. 
\begin{rem}
In \cite{Yasuda:2013fk}, the conjecture above was derived as a direct
consequence of a conjectural change of variables formula of motivic
integrals for the morphism $[\AA_{\cO_{K}}^{n}/\Gamma]\to X$, just
as in the tame case. The weight function $\bw_{\tau}$ appears in
the change of variables formula as necessary shifts coming from twists
of twisted arcs. To compute shifts, Yasuda used a technique which
he calls \emph{untwisting }and which reduces a study of twisted arcs
to one of ordinary arcs. The tuning submodule $\Xi_{\rho}$ is\emph{
}a key ingredient in this technique. This is why it appears in the
definition of the weight.
\begin{rem}
Even if a (crepant) resolution of $X$ does not exist, invariants
$\int_{\Gamma\text{-}\mathrm{Cov}(K)}\LL^{\bw_{\tau}}$ and $M(K,\Gamma,-\bw_{\tau})$
have an important meaning. Indeed they should be equal to stringy
invariants of the quotient singularity and contain rich information
on the geometry of the singularity. 
\end{rem}
\end{rem}
\end{rem}
As mentioned above, if $\kappa$ is algebraically closed and $\sharp\Gamma$
is prime to its characteristic, then $\Gamma\text{-}\mathrm{Cov}(K)$
consists of finitely many points, which correspond to conjugacy classes
of $\Gamma$. If moreover $K=\kappa((t))$ and $\tau$ is defined
over $\kappa$, then we easily see that the conjecture is reduced
to Theorem \ref{thm: McKay tame 1}. When $K=\kappa((t))$ with $\kappa$
a perfect field of characteristic $p>0$, $\Gamma=\ZZ/p$ and $\tau$
is defined over $\kappa$, then the conjecture holds with $\Gamma\text{-}\mathrm{Cov}(K)$
replaced with a similar parameter space \cite{MR3230848}. 

What is the point counting version of the conjecture? The obvious
choice for the left side of the equality is $\sharp Z(\kappa)$. The
one for the right side is not completely clear. For, there is ambiguity
about what kind of space $\Gamma\text{-}\mathrm{Cov}(K)$ is: which
moduli functor it represents, or whether the moduli space is fine
or coarse. However, a reasonable choice would be 
\begin{equation}
\sum_{\alpha\in\Gamma\text{-}\mathrm{Cov}(K)'}\frac{1}{\sharp\Aut(\alpha)}\cdot q^{\bw_{\tau}(\alpha)},\label{eq: point realization}
\end{equation}
where $\Gamma\text{-}\mathrm{Cov}(K)'$ denotes the set of isomorphism
classes of $\Gamma$-torsors over $\Spec K$ and $\Aut(\alpha)$ is
the automorphism group of a $\Gamma$-torsor $\alpha$. The appearance
of the coefficient $1/\sharp\Aut(\alpha)$ is natural, since it is
customary to count objects with weights inverse proportional to the
order of the automorphism group. 
\begin{lem}
Sum (\ref{eq: point realization}) is equal to $M(K,\Gamma,-\bw_{\tau})$.\end{lem}
\begin{proof}
If $H$ is the stabilizer group of a connected component of a $\Gamma$-torsor
$\alpha$, then $\Aut(\alpha)$ is the opposite group of the centralizer
$C_{\Gamma}(H)$ and there are exactly $\sharp\Gamma/\sharp C_{\Gamma}(H)$
elements of $S_{K,\Gamma}$ corresponding to $\alpha$, which shows
the lemma.
\end{proof}
We pose the following conjecture as the point counting realization
of the last conjecture:
\begin{conjecture}
\label{conj: McKay points}In addition to the assumption of Conjecture
\ref{conj: McKay motivic}, we suppose that the residue field $\kappa$
is finite. Then 
\[
\sharp Z(\kappa)=M(K,\Gamma,-\bw_{\tau}).
\]

\end{conjecture}
Again, if $K=\kappa((t))$ with $\kappa$ a perfect field of characteristic
$p>0$, $\Gamma=\ZZ/p$ and $\tau$ is defined over $\kappa$, then
Conjecture \ref{conj: McKay points} holds \cite[Corollary 6.28]{MR3230848}. 
\begin{lem}
We have
\[
M(K,\Gamma,-\bw_{\tau})=\sum_{[g]\in\cC_{\Gamma,q}}q^{\bw_{\tau}(g)}.
\]
\end{lem}
\begin{proof}
The tame absolute Galois group of $\kappa((t))$ is profinitely generated
by two elements, say $a$ and $b$, with exactly one relation: $bab^{-1}=a^{q}$
\cite[page 410]{MR2392026}. Therefore there exists a bijection
\begin{align*}
S_{K,\Gamma} & \to\{(g,h)\in\Gamma^{2}\mid hgh^{-1}=g^{q}\}\\
\rho & \mapsto(\rho(a),\rho(b)).
\end{align*}
Since $a$ profinitely generates the inertia group, from Lemma \ref{lem: v tame explicit},
the map 
\[
S_{K,\Gamma}\to\Conj{\Gamma},\,\rho\mapsto[\rho(a)]
\]
is compatible with the functions both denoted by $\bw_{\tau}$ in
Definitions \ref{def: weight} and \ref{def: weight matrix}. It follows
that 
\[
\begin{split}M(K,\Gamma,-\bw_{\tau}) & =\frac{1}{\sharp\Gamma}\cdot\sum_{(g,h)\in\Gamma^{2}\text{: }hgh^{-1}=g^{q}}q^{\bw_{\tau}(g)}.\end{split}
\]
For each $g\in\Gamma$, there exist exactly $\sharp C_{\Gamma}(g)$
elements $h\in\Gamma$ such that $hgh^{-1}=g^{q}$. Hence 
\[
\begin{split}M(K,\Gamma,-\bw_{\tau}) & =\frac{1}{\sharp\Gamma}\cdot\sum_{g\in\Gamma\text{: }[g]=[g^{q}]}\sharp C_{\Gamma}(g)\cdot q^{\bw_{\tau}(g)}\\
 & =\sum_{[g]\in\cC_{\Gamma,q}}q^{\bw_{\tau}(g)}.
\end{split}
\]

\end{proof}
This lemma and Corollary \ref{cor: counting tame 1} show:
\begin{cor}
\label{cor: point tame 2}Suppose that $\kappa$ is a perfect field
of characteristic prime to $\sharp\Gamma$, that $K=\kappa((t))$,
and that $\tau$ factors through $\GL n{\kappa}$. Then Conjecture
\ref{conj: McKay points} holds.
\end{cor}
We note that a closely related result was proved in \cite{MR2385245}.
\begin{example}
Let $\Gamma=\ZZ/3$, $K$ a local field with a finite residue field
$\kappa=\FF_{q}$ of characteristic $\ne3$, $\tau:\Gamma\to\GL 3{\cO_{K}}$
the regular representation. From \cite[Example 6.1]{MR2354797} and
Corollary \ref{cor:compare tame w a}, we have
\[
M(K,\Gamma,-\bw_{\tau})=\begin{cases}
2q+1 & (q\equiv1\mod3)\\
1 & (q\equiv2\mod3).
\end{cases}
\]

Let $X:=\AA_{\cO_{K}}^{3}$ and $f:Y\to X$ the blowup along the zero
section $\Spec\cO_{K}\subset X$. Then $f$ is a crepant resolution.
Its pull-back by $\Spec\kappa\hookrightarrow\Spec\cO_{K}$, $f_{\kappa}:Y_{\kappa}\to X_{\kappa}$,
depends only on $\kappa$ (independent of $K$). Therefore the corollary
above shows that 
\[
\sharp f^{-1}(0)(\kappa)=M(K,\Gamma,-\bw_{\tau}).
\]

If $q\equiv1\mod3$, then indeed $X$ is a normal toric 3-fold and
isomorphic to the product of the affine toric surface with $A_{2}$-singularity
and the affine line. Hence the exceptional locus of $f_{\kappa}$
is two projective lines meeting at a point, and has $2q+1$ $\FF_{q}$-points.
Conversely, if $q\equiv2\mod3$, then $Z$ has only one $\FF_{q}$-point.
It shows that $Z$ is irreducible. For, if it was not the case, then
an irreducible component would be a Brauer-Severi variety having a
rational point and hence isomorphic to the projective line (see \cite{MR657430}),
which has more than one $\FF_{q}$-point.
\end{example}

\subsection{The punctual Hilbert scheme}

Though it is hardly known when a crepant resolution exists in the
wild case, there is an infinite series of examples. Let $\AA_{R}^{2}$
be an affine plane over a ring $R$. The Hilbert scheme of $n$ points
of $\AA_{R}^{2},$ 
\[
H_{R}:=\Hilb^{n}(\AA_{R}^{2}/\Spec R),
\]
is a projective smooth $R$-scheme of dimension $2n$. Let 
\[
X_{R}:=\AA_{R}^{2}\times_{R}\cdots\times_{R}\AA_{R}^{2}/S_{n}
\]
be the $n$th symmetric product of $\AA_{R}^{2}$. We can also think
of $X$ as the quotient variety associated to $2\sigma:=\sigma\oplus\sigma$
with $\sigma$ the defining representation of $S_{n}$ (over $R$).
Then there is the Hilbert-Chow morphism 
\[
H_{R}\to X_{R}
\]
(for instance, see \cite{MR1440180}). Under a reasonable assumption
on $R$, in particular, if $R=\cO_{K}$ for a local field $K$, then
from \cite{0537.53056,MR1825408,MR2107324}, this is a crepant resolution.
Note that $H_{R}$ and $X_{R}$ are compatible with the base change
by any ring homomorphism $R\to R'$.

The following theorem provides further strong evidence for Conjectures
\ref{conj: McKay motivic} and \ref{conj: McKay points}.
\begin{thm}
\label{thm: Hilb}Suppose that $K$ is a local field with residue
field $\kappa=\FF_{q}$ and put $R:=\cO_{K}$. Then Conjecture \ref{conj: McKay points}
holds for the crepant resolution $H_{R}\to X_{R}$ and the representation
$2\sigma:S_{n}\to\GL{2n}R$. Moreover, if $Z$ is the preimage of
the origin $0\in X(\kappa)$, then we have
\begin{equation}
\sharp Z(\kappa)=M(K,\Gamma,-\bw_{2\sigma})=\sum_{m=0}^{n}P(n,n-m)q^{m}.\label{eq:McKay S_n}
\end{equation}
\end{thm}
\begin{proof}
From \cite{MR870732} (see also \cite{MR2492446} for the case of
an arbitrary base field), $Z$ is stratified into finitely many affine
spaces:
\[
Z=\bigsqcup_{i=1}^{l}\AA_{\kappa}^{n_{i}}.
\]
Moreover, for each $0\le m\le n-1$, there are exactly $P(n,n-m)$
strata of dimension $m$. In particular, we have
\begin{equation}
\sharp Z(\kappa)=\sum_{m=0}^{n-1}P(n,n-m)q^{m}.\label{eq:number points partitions}
\end{equation}

Let $E=\prod_{i=1}^{l}E_{i}$ be an \'etale extension of degree $n$
with $E_{i}$ fields and let $e_{i}$ and $f_{i}$ be the ramification
and inertia indices of $E_{i}/K$. Then we have the partition of $n$
into exactly $\sum f_{i}$ parts, 
\[
n=\overset{f_{1}\text{ times}}{\overbrace{e_{1}+\cdots+e_{1}}}+\cdots+\overset{f_{l}\text{ times}}{\overbrace{e_{l}+\cdots+e_{l}}}.
\]
We will denote the partition by $\bp(E)$. When $\rho\in S_{K,S_{n}}$
corresponds to $E$, we put $\bp(\rho):=\bp(E)$. Then Bhargava's
formula \cite[Proposition 2.2]{MR2354798} is written as follows,
according to Kedlaya's interpretation: for a partition $\bp$ of $n$
into exactly $n-m$ parts, 
\[
\frac{1}{\sharp\Gamma}\sum_{\substack{\rho\in S_{K,S_{n}}\\
\bp(\rho)=\bp
}
}q^{-\ba_{\sigma}(\rho)}=q^{-m}.
\]
On the other hand, for $\rho$ with $\bp(\rho)=\bp$, 
\[
\bt_{\sigma}(\rho)=m.
\]
Indeed, the induced \'etale extension $E^{\nr}/K^{\nr}$ is isomorphic
to $\prod_{i}L_{i}^{f_{i}}$ with $L_{i}/K$ a field extension of
ramification index $e_{i}$. Then the $G_{K^{\nr}}$-action on $\{1,\dots,n\}$
induced by $\rho$ has $\sum f_{i}$ orbits. Since $\sum f_{i}=n-m$,
we have 
\[
\bt_{\sigma}(\rho)=\codim\,(\kappa)^{G_{K^{\nr}}}=m.
\]

If we denote by $\sharp\bp$ the number of parts of $\bp$, then,
from Corollary \ref{cor:comparison twice}, we have 
\begin{align*}
M(K,S_{n},-\bw_{2\sigma}) & =\frac{1}{n!}\sum_{\bp}\sum_{\rho:\bp(\rho)=\bp}q^{\bw_{2\sigma}(\rho)}\\
 & =\frac{1}{n!}\sum_{\bp}\sum_{\rho:\bp(\rho)=\bp}q^{2\bt_{\sigma}(\rho)-\ba_{\sigma}(\rho)}\\
 & =\sum_{\bp}q^{n-\sharp\bp}\\
 & =\sum_{m=0}^{n-1}P(n,n-m)q^{m}.
\end{align*}
We have proved the theorem.
\end{proof}
\bibliographystyle{alpha}
\bibliography{/Users/highernash/Dropbox/Math_Articles/mybib}

\end{document}